\documentclass[draft]{amsart}

\usepackage{amssymb,amsmath,mathrsfs,graphics,graphicx,mathptm,float,lscape}
\usepackage[T1]{fontenc}
\usepackage[english,francais]{babel}
\usepackage[all]{xy}
\usepackage{lscape}
\usepackage[colorlinks, linktocpage, citecolor = blue, linkcolor = blue]{hyperref}

\theoremstyle{plain}
\newtheorem{thm}{Theorem}[section]
\newtheorem{pro}[thm]{Proposition}
\newtheorem{lem}[thm]{Lemma}
\newtheorem{cor}[thm]{Corollary}
\newtheorem{theoalph}{Theorem}

\theoremstyle{definition}

\newtheorem{eg}[thm]{Example}

\newtheorem{rem}[thm]{Remark}

\newtheorem*{remalph}{Remark}
\newtheorem*{egsalph}{Examples}

\def\og{\leavevmode\raise.3ex\hbox{$\scriptscriptstyle\langle\!\langle$~}}
\def\fg{\leavevmode\raise.3ex\hbox{~$\!\scriptscriptstyle\,\rangle\!\rangle$}}

\begin{document}

\selectlanguage{english}


\title[On solvable subgroups of the Cremona group]{On solvable subgroups of \\
the Cremona group}

\author{Julie \textsc{D\'eserti}}
\address{Institut de Math\'ematiques de Jussieu-Paris Rive gauche, UMR $7586$, Universit\'e 
Paris $7$, B\^atiment Sophie Germain, Case $7012$, $75205$ Paris Cedex 13, France.}

\email{deserti@math.univ-paris-diderot.fr}

\thanks{Author partially supported by ANR Grant "BirPol"  ANR-11-JS01-004-01.}

\begin{abstract}
The Cremona group $\mathrm{Bir}(\mathbb{P}^2_\mathbb{C})$ is the group of birational self-maps of $\mathbb{P}^2_\mathbb{C}$. Using the action of $\mathrm{Bir}(\mathbb{P}^2_\mathbb{C})$ on the Picard-Manin space of $\mathbb{P}^2_\mathbb{C}$ we characterize its solvable subgroups. If $\mathrm{G}\subset\mathrm{Bir}(\mathbb{P}^2_\mathbb{C})$ is solvable, non virtually abelian, and infinite, then up to finite index: either any element of $\mathrm{G}$ is of finite order or conjugate to an automorphism of $\mathbb{P}^2_\mathbb{C}$, or $\mathrm{G}$ preserves a unique fibration that is rational or elliptic, or $\mathrm{G}$ is, up to conjugacy, a subgroup of the group generated by one hyperbolic monomial map and the diagonal automorphisms. 

We also give some corollaries.
\end{abstract}

\maketitle

\section{Introduction}\label{sec:intro}

We know properties on finite subgroups (\cite{DolgachevIskovskikh}), finitely generated subgroups (\cite{Cantat:tits}), uncountable maximal abelian subgroups (\cite{Deserti:autbir}), nilpotent subgroups (\cite{Deserti:nilpotent}) of the Cremona group. In this article we focus on solvable subgroups of the Cremona group.

\medskip

Let $\mathrm{G}$ be a group. Recall that $[g,h]=ghg^{-1}h^{-1}$ denotes the commutator of $g$ and $h$. If $\Gamma_1$ and $\Gamma_2$ are two subgroups of $\mathrm{G}$, then $[\Gamma_1,\Gamma_2]$ is the subgroup of $\mathrm{G}$ generated by the elements of the form $[g,h]$ with $g\in \Gamma_1$ and $h\in \Gamma_2$. We define the \textbf{\textit{derived series}} of $\mathrm{G}$ by setting
\[
\mathrm{G}^{(0)}=\mathrm{G}\qquad\qquad \mathrm{G}^{(n+1)}=[\mathrm{G}^{(n)},\mathrm{G}^{(n)}] \quad\forall\,n\geq 0.
\]
The \textbf{\textit{soluble length}} $\ell(\mathrm{G})$ of $\mathrm{G}$ is defined by  
\[
\ell(\mathrm{G})=\min\big\{k\in\mathbb{N}\cup\{0\}\,\vert\, \mathrm{G}^{(k)}=\{\mathrm{id}\}\big\}
\]
with the convention: $\min \emptyset=\infty$. We say that $\mathrm{G}$ is \textbf{\textit{solvable}} if $\ell(\mathrm{G})<\infty$. The study of solvable groups started a long time ago, and any linear solvable subgroup is up to finite index triangularizable (Lie-Kolchin theorem, \cite[Theorem 21.1.5]{KargapolovMerzljakov}). The assumption "up to finite index" is essential: for instance the subgroup of $\mathrm{PGL}(2,\mathbb{C})$ generated by $\left[\begin{array}{cc}
1 & 0\\
1 & -1
\end{array}\right]$ and
$\left[\begin{array}{cc}
-1 & 1\\
0 & 1 
\end{array}\right]$ is isomorphic to $\mathfrak{S}_3$ so is solvable but is not triangularizable.

\medskip

\begin{theoalph}\label{mainthm}
{\sl Let $\mathrm{G}$ be an infinite, solvable, non virtually abelian subgroup of $\mathrm{Bir}(\mathbb{P}^2_\mathbb{C})$. Then, up to finite index, one of the following holds:
\smallskip
\begin{enumerate}
\item any element of $\mathrm{G}$ is either of finite order, or conjugate to an automorphism of $\mathbb{P}^2_\mathbb{C}$;
\smallskip
\item $\mathrm{G}$ preserves a unique fibration that is rational, in particular\, $\mathrm{G}$ is, up to conjugacy, a subgroup of\,~$\mathrm{PGL}(2,\mathbb{C}(y))~\rtimes~\mathrm{PGL}(2,\mathbb{C})$;
\smallskip
\item $\mathrm{G}$ preserves a unique fibration that is elliptic;
\smallskip
\item $\mathrm{G}$ is, up to birational conjugacy, contained in the group generated by 
\[
\big\{ (x^py^q,x^ry^s),\,(\alpha x,\beta y)\,\vert\, \alpha,\,\beta\in\mathbb{C}^*\big\}
\]
where $M=\left[\begin{array}{cc}
p & q\\
r & s
\end{array}\right]$ denotes an element of $\mathrm{GL}(2,\mathbb{Z})$ with spectral radius $>1$. The group $\mathrm{G}$ preserves the two holomorphic foliations defined by the $1$-forms $\alpha_1\,x\mathrm{d}y+\beta_1\,y\mathrm{d}x$ and $\alpha_2\,x\mathrm{d}y+\beta_2\,y\mathrm{d}x$ where $(\alpha_1,\beta_1)$ and $(\alpha_2,\beta_2)$ denote the eigenvectors of\, ${}^{t}M$.
\end{enumerate}
\smallskip
Furthermore if $\mathrm{G}$ is uncountable, case 3. does not hold.}
\end{theoalph}

\begin{egsalph}\label{exemples}
\begin{itemize}
\item[$\bullet$] Denote by $\mathrm{S}_3$ the group generated by $\left[\begin{array}{cc}
1 & 0\\
1 & -1
\end{array}\right]$ and
$\left[\begin{array}{cc}
-1 & 1\\
0 & 1 
\end{array}\right]$. As we recall before $\mathrm{S}_3\simeq\mathfrak{S}_3$. Consider now the subgroup $\mathrm{G}$ of $\mathrm{Bir}(\mathbb{P}^2_\mathbb{C})$ whose elements are the monomial maps $(x^py^q,x^ry^s)$ with $\left[
\begin{array}{cc}
p & q \\
r & s
\end{array}
\right]\in\mathrm{S}_3$. Then any element of $\mathrm{G}$ has finite order, and $\mathrm{G}$ is solvable; it gives an example of case~1. 

\smallskip

\item[$\bullet$] Other examples that illustrate case 1. are the following groups 
\[
\big\{(\alpha x+\beta y+\gamma,\delta y+\varepsilon)\,\vert\,\alpha,\,\delta\in\mathbb{C}^*,\,\beta,\,\gamma,\,\varepsilon\in\mathbb{C}\big\}\subset\mathrm{Aut}(\mathbb{P}^2_\mathbb{C}),
\]
and
\[
\mathrm{E}=\big\{(\alpha x+P(y),\beta y+\gamma)\,\vert\, \alpha,\,\beta\in\mathbb{C}^*,\,\gamma\in\mathbb{C},\, P\in\mathbb{C}[y]\big\}\subset\mathrm{Aut}(\mathbb{C}^2).
\]

\item[$\bullet$] The centralizer of a birational map of $\mathbb{P}^2_\mathbb{C}$ that preserves a unique fibration that is rational is virtually solvable (\cite[Corollary C]{CerveauDeserti:centralisateurs}); this example falls in case 2 (\emph{see} \S\ref{subsec:parabolic}).

\item[$\bullet$] In \cite[Proposition 2.2]{Cornulier} Cornulier proved that the group
\[
\langle (x+1,y),\,(x,y+1),\,(x,xy)\rangle
\]
is solvable of length $3$, and is not linear over any field; this example falls in case 2. The invariant fibration is given by $x=$ cst.
\end{itemize}
\end{egsalph}

\begin{remalph}
In case $1.$ if there exists an integer $d$ such that $\deg\phi\leq d$ for any $\phi$ in $\mathrm{G}$, then there exist a smooth projective variety $M$ and a birational map $\psi\colon M\dashrightarrow\mathbb{P}^2_\mathbb{C}$ such that $\psi^{-1}\mathrm{G}\psi$ is a solvable subgroup of $\mathrm{Aut}(M)$ (\emph{see} \S\ref{subsec:elliptic}). But there is some solvable subgroups $\mathrm{G}$ with only elliptic elements that do not satisfy this property: the group $\mathrm{E}$ introduced in Examples \ref{exemples}. Let us mention an other example: Wright constructs abelian subgroups $\mathrm{H}$ of $\mathrm{Aut}(\mathbb{C}^2)$ such that any element of $\mathrm{H}$ is of finite order, $\mathrm{H}$ is unbounded and does not preserve any fibration (\cite{Wright}).
\end{remalph}

In \S \ref{sec:proofmainthm} we prove Theorem \ref{mainthm}: we first assume that our solvable, infinite and non virtually abelian subgroup $\mathrm{G}$ contains a hyperbolic map, then that it contains a twist and no hyperbolic map, and finally that all elements of $\mathrm{G}$ are elliptic. In the last section (\S\ref{sec:appl}) we also
\smallskip
\begin{itemize}
\item[$\bullet$] recover the following fact: if $\mathrm{G}$ is an infinite nilpotent subgroup of $\mathrm{Bir}(\mathbb{P}^2_\mathbb{C})$, then $\mathrm{G}$ does not contain a hyperbolic map;

\item[$\bullet$] remark that we can bound the soluble length of a nilpotent subgroup of $\mathrm{Bir}(\mathbb{P}^2_\mathbb{C})$ by the dimension of~$\mathbb{P}^2_\mathbb{C}$ as Epstein and Thurston did in the context of Lie algebras of rational vector fields on a connected complex manifold;
\smallskip
\item[$\bullet$] give a negative answer to the following question of Favre: does any solvable and finitely generated subgroup $\mathrm{G}$ of $\mathrm{Bir}(\mathbb{P}^2_\mathbb{C})$ contain a subgroup of finite index whose commutator subgroup is nilpotent ? if we assume that $[\mathrm{G},\mathrm{G}]$ is not a torsion group;
\smallskip
\item[$\bullet$] give a description of the embeddings of the solvable Baumslag-Solitar groups into the Cremona group.
\end{itemize}

\subsection*{Acknowledgments} I would like to thank D. Cerveau for always having time to listen to me, for his precious advices, and for pointing out \cite{MarteloRibon}. Thanks to S. Cantat for helpful remarks and comments, and also to J. Ribon for the example of \S\ref{sec:favre}. I would like also thank P. Sad for his permanent invitation at IMPA where I found the quietness to put my ideas in order. Thanks to the referee whose comments and suggestions help me to improve the presentation.

\section{Some properties of the birational maps}\label{sec:someptes}

\subsubsection*{First definitions}

Let $\mathcal{S}$ be a projective surface. We will denote by $\mathrm{Bir}(\mathcal{S})$ the group of birational self-maps of $\mathcal{S}$; in the particular case of the complex projective plane the group $\mathrm{Bir}(\mathbb{P}^2_\mathbb{C})$ is called \textbf{\textit{Cremona group}}. Take $\phi$ in $\mathrm{Bir}(\mathcal{S})$, we will denote by $\mathrm{Ind}\,\phi$ the set of points of indeterminacy of $\phi$; the codimension of $\mathrm{Ind}\,\phi$ is $\geq 2$.

A birational map from $\mathbb{P}^2_\mathbb{C}$ into itself can be written 
\[
(x:y:z)\dashrightarrow\big(\phi_0(x,y,z):\phi_1(x,y,z):\phi_2(x,y,z)\big)
\]
where the $\phi_i$'s denote some homogeneous polynomials of the same degree and without common factors of positive degree. The \textbf{\textit{degree}} of $\phi$ is equal to the degree of the $\phi_i$'s. Let $\phi$ be a birational map of $\mathbb{P}^2_\mathbb{C}$. One can define the \textbf{\textit{dynamical degree}} of $\phi$ as
\[
\lambda(\phi)=\lim_{n\to +\infty}(\deg \phi^n)^{1/n}.
\]
More generally let $\mathcal{S}$ be a projective surface, and $\phi\colon \mathcal{S}\dashrightarrow \mathcal{S}$ be a birational map. Take any norm $\vert\vert\cdot\vert\vert$ on the N\'eron-Severi real vector space $\mathrm{N}^1(\mathcal{S})$. If $\phi^*$ is the induced action by $\phi$ on $\mathrm{N}^1(\mathcal{S})$, we can define
\[
\lambda(\phi)=\lim_{n\to +\infty}\vert\vert (\phi^n)^*\vert\vert^{1/n}.
\]
Remark that this quantity is a birational invariant: if $\psi\colon \mathcal{S}\dashrightarrow \mathcal{S}'$ is a birational map, then $\lambda(\psi\phi\psi^{-1})=\lambda(\phi)$. 

\subsubsection*{Classification of birational maps}

The algebraic degree is not a birational invariant, but the first dyna\-mical degree is; more precisely one has a classification of birational maps based on the degree growth. Before stating it, let us first introduce the following definitions. Let $\phi$ be an element of $\mathrm{Bir}(\mathbb{P}^2_\mathbb{C})$. If
\smallskip
\begin{itemize}
\item[$\bullet $] $(\deg\phi^n)_{n\in\mathbb{N}}$ is bounded, we say that $\phi$ is an \textbf{\textit{elliptic map}}, 
\smallskip
\item[$\bullet $] $(\deg\phi^n)_{n\in\mathbb{N}}$ grows linearly, we say that $\phi$ is a \textbf{\textit{Jonqui\`eres twist}}, 
\smallskip
\item[$\bullet $] $(\deg\phi^n)_{n\in\mathbb{N}}$ grows quadratically, we say that $\phi$ is a \textbf{\textit{Halphen twist}}, 
\smallskip
\item[$\bullet $] $(\deg\phi^n)_{n\in\mathbb{N}}$ grows exponentially, we say that $\phi$ is a \textbf{\textit{hyperbolic map}}. 
\end{itemize}

\begin{thm}[\cite{DillerFavre, Gizatullin, BlancDeserti:degreegrowth}]
{\sl Let $\phi$ be an element of $\mathrm{Bir}(\mathbb{P}^2_\mathbb{C})$. Then one and only one of the following cases holds
\smallskip
\begin{itemize}
\item[$\bullet$] $\phi$ is elliptic, furthermore if $\phi$ is of infinite order, then $\phi$ is up to birational conjugacy an automorphism of $\mathbb {P}^2_\mathbb{C}$,
\smallskip
\item[$\bullet$] $\phi$ is a Jonqui\`eres twist, $\phi$ preserves a unique fibration that is rational and every conjugate of $\phi$ is not an automorphism of a projective surface, 
\smallskip
\item[$\bullet$] $\phi$ is a Halphen twist, $\phi$ preserves a unique fibration that is elliptic and $\phi$ is conjugate to an automorphism of a projective surface, 
\smallskip
\item[$\bullet$] $\phi$ is a hyperbolic map. 
\end{itemize}
\smallskip

In the three first cases, $\lambda(\phi)=1$, in the last one $\lambda(\phi)>1$.}
\end{thm}

\subsubsection*{The Picard-Manin and bubble spaces}

Let $\mathcal{S}$, and $\mathcal{S}_i$ be complex projective surfaces. If $\pi\colon \mathcal{S}_1\to \mathcal{S}$ is a birational morphism, one gets $\pi^*\colon\mathrm{N}^1(\mathcal{S})\to\mathrm{N}^1(\mathcal{S}_1)$ an embedding of N\'eron-Severi groups. Take two birational morphisms $\pi_1\colon \mathcal{S}_1\to \mathcal{S}$ and $\pi_2\colon \mathcal{S}_2\to \mathcal{S};$ the morphism $\pi_2$ is \textbf{\textit{above}} $\pi_1$ if $\pi_1^{-1}\pi_2$ is regular. Starting with two birational morphisms one can always find a third one that covers the two first. Therefore the inductive limit of all groups $\mathrm{N}^1(\mathcal{S}_i)$ for all surfaces $\mathcal{S}_i$ above $\mathcal{S}$ is well-defined; it is the \textbf{\textit{Picard-Manin space}} $\mathcal{Z}_\mathcal{S}$ of $\mathcal{S}$. For any birational map $\pi$, $\pi^*$ preserves the intersection form and maps nef classes to nef classes hence the limit space $\mathcal{Z}_\mathcal{S}$ is endowed with an intersection form of signature $(1,\infty)$ and a nef cone.

\smallskip

Let $\mathcal{S}$ be a complex projective surface. Consider all complex and projective surfaces $\mathcal{S}_i$ above~$\mathcal{S}$, that is all birational morphisms $\pi_i\colon\mathcal{S}_i\to\mathcal{S}$. If $p$ (resp. $q$) is a point of a complex projective surface $\mathcal{S}_1$ (resp.~$\mathcal{S}_2$), and if $\pi_1\colon\mathcal{S}_1\to\mathcal{S}$ (resp. $\pi_2\colon\mathcal{S}_2\to\mathcal{S}$) is a birational morphism, then $p$ is identified with $q$ if~$\pi_1^{-1}\pi_2$ is a local isomorphism in a neighborhood of $q$ that maps $q$ onto $p$. The \textbf{\textit{bubble space}} $\mathcal{B}(\mathcal{S})$ is the union of all points of all surfaces above $\mathcal{S}$ modulo the equivalence relation induced by this identification. If $p$ belongs to $\mathcal{B}(\mathcal{S})$ represented by a point $p$ on a surface $\mathcal{S}_i\to\mathcal{S}$, denote by $E_p$ the exceptional divisor of the blow-up of $p$ and by $e_p$ its divisor class viewed as a point in $\mathcal{Z}_\mathcal{S}$. The following properties are satisfied 
\[
\left\{
\begin{array}{ll}
e_p\cdot e_q=0 \text{ if }p\not=q\\
e_p\cdot e_p=-1
\end{array}
\right.
\]

\subsubsection*{Hyperbolic space $\mathbb{H}_\mathcal{S}$}

Embed $\mathrm{N}^1(\mathcal{S})$ as a subgroup of $\mathcal{Z}_\mathcal{S}$; this finite dimensional lattice is orthogonal to $e_p$ for any $p\in\mathcal{B}(\mathcal{S})$, and 
\[
\mathcal{Z}_\mathcal{S}=\big\{ D+\sum_{p\in\mathcal{B}(\mathcal{S})} a_pe_p\,\vert\, D\in\mathrm{N}^1(\mathcal{S}),\,a_p\in\mathbb{R}\big\}.
\]
The \textbf{\textit{completed Picard-Manin space}} $\overline{\mathcal{Z}}_\mathcal{S}$ of $\mathcal{S}$ is the $L^2$-completion of~$\mathcal{Z}_\mathcal{S}$; in other words
\[
\overline{\mathcal{Z}}_\mathcal{S}=\big\{D+\sum_{p\in\mathcal{B}(\mathcal{S})} a_pe_p\,\vert\, D\in\mathrm{N}^1(\mathcal{S}),\,a_p\in\mathbb{R},\,\sum a_p^2<+\infty\big\}.
\]
The intersection form extends as an intersection form with signature $(1,\infty)$ on $\overline{\mathcal{Z}}_\mathcal{S}$. Let 
\[
\overline{\mathcal{Z}}_{\mathcal{S}}^+=\big\{d\in\overline{\mathcal{Z}}_{\mathcal{S}}\,\vert\, d\cdot c\geq 0 \quad\forall\,c\in \overline{\mathcal{Z}}_{\mathcal{S}}\}
\]
be the nef cone of~$\overline{\mathcal{Z}}_\mathcal{S}$ and
\[
\mathcal{L}\overline{\mathcal{Z}}_{\mathcal{S}}=\big\{d\in\overline{\mathcal{Z}}_{\mathcal{S}}\,\vert\, d\cdot d=0\big\}
\]
be the light cone of $\overline{\mathcal{Z}}_{\mathcal{S}}$.

The \textbf{\textit{hyperbolic space}} $\mathbb{H}_\mathcal{S}$ of $\mathcal{S}$ is then defined by 
\[
\mathbb{H}_\mathcal{S}=\big\{d\in\overline{\mathcal{Z}}_\mathcal{S}^+\,\vert\, d\cdot d=1\big\}.
\] 
Let us remark that $\mathbb{H}_\mathcal{S}$ is an infinite dimensional analogue of the classical hyperbolic space $\mathbb{H}^n$. The \textbf{\textit{distance}} on $\mathbb{H}_\mathcal{S}$ is defined by 
\[
\cosh (\mathrm{dist}(d,d'))=d\cdot d' \qquad\qquad\qquad \forall\,d,\,d'\in\mathbb{H}_\mathcal{S}.
\]
The \textbf{\textit{geodesics}} are intersections of $\mathbb{H}_\mathcal{S}$ with planes. The projection of $\mathbb{H}_\mathcal{S}$ onto $\mathbb{P}(\overline{\mathcal{Z}}_\mathcal{S})$ is one-to-one, and the boundary of its image is the projection of the cone of isotropic vectors of $\overline{\mathcal{Z}}_\mathcal{S}$. Hence 
\[
\partial\mathbb{H}_\mathcal{S}=\big\{\mathbb{R}_+d\,\vert\, d\in\overline{\mathcal{Z}}_\mathcal{S}^+,\, d\cdot d=0\big\}.
\]

\subsubsection*{Isometries of $\mathbb{H}_\mathcal{S}$}

If $\pi\colon\mathcal{S}'\to\mathcal{S}$ is a birational morphism, we get a canonical isometry $\pi^*$ (and not only an embedding) between $\mathbb{H}_{\mathcal{S}}$ and $\mathbb{H}_{\mathcal{S}'}$. This allows to define an action of $\mathrm{Bir}(\mathcal{S})$ on $\mathbb{H}_\mathcal{S}$. Consider a birational map $\phi$ on a complex projective surface $\mathcal{S}$. There exists a surface~$\mathcal{S}'$, and $\pi_1\colon\mathcal{S}'\to\mathcal{S}$, $\pi_2\colon\mathcal{S}'\to\mathcal{S}$ two morphisms such that $\phi=\pi_2\pi_1^{-1}$. One can define $\phi_\bullet$ by
\[
\phi_\bullet=(\pi_2^*)^{-1}\pi_1^*; 
\]
in fact, one gets a faithful representation of $\mathrm{Bir}(\mathcal{S})$ into the group of isometries of $\mathbb{H}_\mathcal{S}$ (\emph{see} \cite{Cantat:tits}).

The isometries of $\mathbb{H}_S$ are classified in three types (\cite{BridsonHaefliger, GhysdelaHarpe}). The \textbf{translation length} of an isometry $\phi_\bullet$ of~$\mathbb{H}_S$ is defined by
\[
L(\phi_\bullet)=\inf\big\{\mathrm{dist}(p,\phi_\bullet(p))\,\vert\, p\in\mathbb{H}_S\big\}
\]
If the infimum is a minimum, then
\begin{itemize}
\item either it is equal to $0$ and $\phi_\bullet$ has a fixed point in $\mathbb{H}_S$, $\phi_\bullet$ is thus \textbf{elliptic},

\item or it is positive and $\phi_\bullet$ is \textbf{hyperbolic}. Hence the set of points $p\in\mathbb{H}_S$ such that $\mathrm{dist}(p,\phi_\bullet(p))$ is equal to $L(\phi_\bullet)$ is a geodesic line $\mathrm{Ax}(\phi_\bullet)\subset\mathbb{H}_S$. Its boundary points are represented by isotropic vectors $\omega(\phi_\bullet)$ and $\alpha(\phi_\bullet)$ in $\overline{\mathcal{Z}}_S$ such that
\[
\phi_\bullet(\omega(\phi_\bullet))=\lambda(\phi)\,\omega(\phi_\bullet)\qquad\qquad\phi_\bullet(\alpha(\phi_\bullet))=\frac{1}{\lambda(\phi)}\alpha(\phi_\bullet)
\]
The axis $\mathrm{Ax}(\phi_\bullet)$ of $\phi_\bullet$ is the intersection of $\mathbb{H}_S$ with the plane containing $\omega(\phi_\bullet)$ and $\alpha(\phi_\bullet)$; furthermore $\phi_\bullet$ acts as a translation of length $L(\phi_\bullet)=\log\lambda(\phi)$ along $\mathrm{Ax}(\phi_\bullet)$ (\emph{see} \cite[Remark 4.5]{CantatLamy}). For all~$p$ in $\mathbb{H}_S$ one has 
\[
\lim_{k\to +\infty}\frac{\phi_\bullet^{-k}(p)}{\lambda(\phi)}=\alpha(\phi_\bullet)\qquad\qquad \lim_{k\to +\infty}\frac{\phi_\bullet^k(p)}{\lambda(\phi)}=\omega(\phi_\bullet).
\]
\end{itemize}

When the infimum is not realized, $L(\phi_\bullet)=0$ and $\phi_\bullet$ is \textbf{parabolic}: $\phi_\bullet$ fixes a unique line in $\mathcal{L}\overline{\mathcal{Z}}_S$; this line is fixed pointwise, and all orbits $\phi_\bullet^n(p)$ in $\mathbb{H}_S$ accumulate to the corresponding boundary point when $n$ goes to $\pm\infty$.

\smallskip

There is a strong relationship between this classification and the classification of birational maps of the complex projective plane (\cite[Theorem 3.6]{Cantat:tits}): if $\phi$ is an element of $\mathrm{Bir}(\mathbb{P}^2_\mathbb{C})$, then
\smallskip
\begin{itemize}
\item[$\bullet$] $\phi_\bullet$ is an elliptic isometry if and only if $\phi$ is an elliptic map;
\smallskip
\item[$\bullet$] $\phi_\bullet$ is a parabolic isometry if and only if $\phi$ is a twist;
\smallskip
\item[$\bullet$] $\phi_\bullet$ is a hyperbolic isometry if and only if $\phi$ is a hyperbolic map.
\end{itemize}

\subsubsection*{Tits alternative}

Cantat proved the Tits alternative for the Cremona group (\cite[Theorem C]{Cantat:tits}): let $\mathrm{G}$ be a finitely generated subgroup of $\mathrm{Bir}(\mathbb{P}^2_\mathbb{C})$, then
\smallskip
\begin{itemize}
\item[$\bullet$] either $\mathrm{G}$ contains a free non abelian subgroup,
\smallskip
\item[$\bullet$] or $\mathrm{G}$ contains a subgroup of finite index that is solvable.
\end{itemize}
\smallskip

As a consequence he studied finitely generated and solvable subgroups of $\mathrm{Bir}(\mathbb{P}^2_\mathbb{C})$ without torsion (\cite[Theorem 7.3]{Cantat:tits}): let $\mathrm{G}$ be such a group, there exists a subgroup $\mathrm{G}_0$ of $\mathrm{G}$ of finite index such that
\smallskip
\begin{itemize}
\item[$\bullet$] either $\mathrm{G}_0$ is abelian,
\smallskip
\item[$\bullet$] or $\mathrm{G}_0$ preserves a foliation. 
\end{itemize}
\smallskip

\section{Proof of Theorem \ref{mainthm}}\label{sec:proofmainthm}

\subsection{Solvable groups of birational maps containing a hyperbolic map}\label{subsec:hyperbolic}

Let us recall the following criterion (for its proof see for example \cite{delaHarpe}) used on many occasions by Klein, and also by Tits (\cite{Tits}) known as Ping-Pong Lemma: {\sl let $\mathrm{H}$ be a group acting on a set $X$, let $\Gamma_1$, $\Gamma_2$ be two subgroups of~$\mathrm{H}$, and let $\Gamma$ be the subgroup generated by $\Gamma_1$ and $\Gamma_2$. Assume that $\Gamma_1$ contains at least three elements, and $\Gamma_2$ at least two elements. Suppose that there exist two non-empty subsets $X_1$, $X_2$ of $X$ such that 
\[
X_2\not\subset X_1,\qquad\gamma(X_2)\subset X_1\quad \forall\,\gamma\in\Gamma_1\smallsetminus\{\mathrm{id}\}, \qquad
\gamma^{\,\prime}(X_1)\subset X_2\quad \forall\,\gamma^{\,\prime}\in\Gamma_2\smallsetminus\{\mathrm{id}\}. 
\]
Then $\Gamma$ is isomorphic to the free product $\Gamma_1\ast\Gamma_2$.} The Ping-Pong argument allows us to prove the following:

\begin{lem}\label{lem:pas2hyp}
{\sl A solvable, non abelian subgroup of $\mathrm{Bir}(\mathbb{P}^2_\mathbb{C})$ cannot contain two hyperbolic maps $\phi$ and~$\psi$ such that $\big\{\omega(\phi_\bullet),\,\alpha(\phi_\bullet)\big\}\not=\big\{\omega(\psi_\bullet),\,\alpha(\psi_\bullet)\big\}$.}
\end{lem}

\begin{proof}
Assume by contradiction that $\big\{\omega(\phi_\bullet),\,\alpha(\phi_\bullet)\big\}\not=\big\{\omega(\psi_\bullet),\,\alpha(\psi_\bullet)\big\}$. Then the Ping-Pong argument implies that there exist two integers $n$ and $m$ such that $\psi^n$ and $\phi^m$ generate a subgroup of $\mathrm{G}$ isomorphic to the free group $\mathrm{F}_2$ (\emph{see} \cite{Cantat:tits}). But $\langle\phi,\,\psi\rangle$ is a solvable group: contradiction.
\end{proof}

Let $\mathrm{G}$ be an infinite solvable, non virtually abelian, subgroup of $\mathrm{Bir}(\mathbb{P}^2_\mathbb{C})$. Assume that $\mathrm{G}$ contains a hyperbolic map $\phi$. Let $\alpha(\phi_\bullet)$ and~$\omega(\phi_\bullet)$ be the two fixed points of $\phi_\bullet$ on $\partial\mathbb{H}_{\mathbb{P}^2_\mathbb{C}}$, and $\mathrm{Ax}(\phi_\bullet)$ be the geodesic passing through these two points. As $\mathrm{G}$ is solvable there exists a subgroup of $\mathrm{G}$ of index $\leq 2$ that preserves $\alpha(\phi_\bullet)$, $\omega(\phi_\bullet)$, and $\mathrm{Ax}(\phi_\bullet)$ (\emph{see} \cite[Theorem 6.4]{Cantat:tits}); let us still denote by~$\mathrm{G}$ this subgroup. Note that there is no twist in $\mathrm{G}$ since a parabolic isometry has a unique fixed point on $\partial\mathbb{H}_{\mathbb{P}^2_\mathbb{C}}$. One has a morphism $\kappa\colon\mathrm{G}\to\mathbb{R}_{>0}$ such that 
\[
\psi_\bullet(\ell)=\kappa(\psi)\ell
\]
for any $\ell$ in $\overline{\mathcal{Z}}_{\mathbb{P}^2_\mathbb{C}}$ lying on $\mathrm{Ax}(\phi_\bullet)$.

The kernel of $\kappa$ is an infinite subgroup that contains only elliptic maps. Indeed the set of elliptic elements of $\mathrm{G}$ coincides with $\ker\kappa$; and $[\mathrm{G},\mathrm{G}]\subset \ker\kappa$ so if $\ker\kappa$ is finite, $\mathrm{G}$ is abelian up to finite index which is by assumption impossible. 

\subsubsection*{Gap property}

If $\psi$ is an hyperbolic birational map of $\mathrm{G}$, then $\kappa(\psi)=L(\psi_\bullet)=\log \lambda(\psi)$. Recall that $\lambda(\phi)$ is an algebraic integer with all Galois conjugates in the unit disk, that is a Salem number, or a Pisot number. The smallest known number is the Lehmer number $\lambda_L\simeq 1,176$ which is a root of $X^{10}+X^9-X^7-X^6-X^5-X^4-X^3+X+1$. Blanc and Cantat prove in \cite[Corollary 2.7]{BlancCantat} that there is a gap in the dynamical spectrum $\Lambda=\big\{\lambda(\phi)\,\vert\,\phi\in\mathrm{Bir}(\mathbb{P}^2_\mathbb{C})\}$: there is no dynamical degree in $]1,\lambda_L[$.

The gap property implies that in fact $\kappa$ is a morphism from $\mathrm{G}$ to a subgroup of $\mathbb{R}_{>0}$ isomorphic to $\mathbb{Z}$.

\subsubsection*{Elliptic subgroups of the Cremona group with a large normalizer}

Consider in $\mathbb{P}^2_\mathbb{C}$ the complement of the union of the three lines $\{x=0\}$, $\{y=0\}$ and $\{z=0\}$. Denote by $\mathcal{U}$ this open set isomorphic to $\mathbb{C}^*\times\mathbb{C}^*$. One has an action of $\mathbb{C}^*\times\mathbb{C}^*$ on $\mathcal{U}$ by translation. Furthermore~$\mathrm{GL}(2,\mathbb{Z})$ acts on $\mathcal{U}$ by monomial maps
\[
\left[
\begin{array}{cc}
p & q\\
r & s
\end{array}
\right] \mapsto \big((x,y)\mapsto(x^py^q,x^ry^s)\big)
\]
One thus has an injective morphism from $(\mathbb{C}^*\times\mathbb{C}^*)\rtimes\mathrm{GL}(2,\mathbb{Z})$ into $\mathrm{Bir}(\mathbb{P}^2_\mathbb{C})$. Let $\mathrm{G}_{\textrm{toric}}$ be its image.

One can now apply \cite[Theorem 4]{DelzantPy} that says that if there exists a short exact sequence
\[
1 \longrightarrow \mathrm{A}\longrightarrow \mathrm{N}\longrightarrow \mathrm{B}\longrightarrow 1
\]
where $\mathrm{N}\subset\mathrm{Bir}(\mathbb{P}^2_\mathbb{C})$ contains at least one hyperbolic element, and $\mathrm{A}\subset\mathrm{Bir}(\mathbb{P}^2_\mathbb{C})$ is an infinite and elliptic\footnote{A subgroup of $\mathrm{Bir}(\mathbb{P}^2_\mathbb{C})$ is elliptic if it fixes a point in $\mathbb{H}_{\mathbb{P}^2_\mathbb{C}}$.} group, then $\mathrm{N}$ is up to conjugacy a subgroup of $\mathrm{G}_{\textrm{toric}}$. Hence up to birational conjugacy $\mathrm{G}\subset\mathrm{G}_{\textrm{toric}}$. Recall now that if $\psi$ is a hyperbolic map of the form $(x^ay^b,x^cy^d)$, then to preserve $\alpha(\psi_\bullet)$ and $\omega(\psi_\bullet)$ is equivalent to preserve the eigenvectors of the matrix $\left[
\begin{array}{cc}
a & b\\
c & d
\end{array}
\right]$. We can now thus state: 

\begin{pro}\label{pro:1hyp}
{\sl Let $\mathrm{G}$ be an infinite solvable, non virtually abelian, subgroup of $\mathrm{Bir}(\mathbb{P}^2_\mathbb{C})$. If $\mathrm{G}$ contains a hyperbolic birational map, then $\mathrm{G}$ is, up to conjugacy and finite index, a subgroup of the group generated by
\[
\big\{ (x^py^q,x^ry^s),\,(\alpha x,\beta y)\,\vert\, \alpha,\,\beta\in\mathbb{C}^*\big\}
\]
where $\left[\begin{array}{cc}
p & q\\
r & s
\end{array}\right]$ denotes an element of $\mathrm{GL}(2,\mathbb{Z})$ with spectral radius $>1$.
}
\end{pro}

\subsection{Solvable groups with a twist}\label{subsec:parabolic}

Consider a solvable, non abelian subgroup $\mathrm{G}$ of $\mathrm{Bir}(\mathbb{P}^2_\mathbb{C})$. Let us assume that~$\mathrm{G}$ contains a twist $\phi$; the map $\phi$ preserves a unique fibration~$\mathcal{F}$ that is rational or elliptic. Let us prove that any element of $\mathrm{G}$ preserves $\mathcal{F}$. Denote by $\alpha(\phi_\bullet)\in\partial\mathbb{H}_{\mathbb{P}^2_\mathbb{C}}$ the fixed point of $\phi_\bullet$. Take one element in $\mathcal{L}\overline{\mathcal{Z}}_{\mathbb{P}^2_\mathbb{C}}$ still denoted $\alpha(\phi_\bullet)$ that represents $\alpha(\phi_\bullet)$. Take $\varphi\in\mathrm{G}$ such that $\varphi(\alpha(\phi_\bullet))\not=\alpha(\phi_\bullet)$. Then $\psi=\varphi\phi\varphi^{-1}$ is parabolic and fixes the unique element $\alpha(\psi_\bullet)$ of $\mathcal{L}\overline{\mathcal{Z}}_{\mathbb{P}^2_\mathbb{C}}$ proportional to $\varphi(\alpha(\phi_\bullet))$. Take $\varepsilon>0$ such that $\mathcal{U}(\alpha(\phi_\bullet),\varepsilon)\cap\mathcal{U}(\alpha(\psi_\bullet),\varepsilon)=\emptyset$ where
\[
\mathcal{U}(\alpha,\varepsilon)=\big\{\ell\in\mathcal{L}\overline{\mathcal{Z}}_{\mathbb{P}^2_\mathbb{C}}\,\vert\,\alpha\cdot\ell<\varepsilon\big\}.
\]
Since $\psi_\bullet$ is parabolic, then for $n$ large enough $\psi_\bullet^n\big(\mathcal{U}(\alpha(\phi_\bullet),\varepsilon)\big)$ is included in a $\mathcal{U}(\alpha(\psi_\bullet),\varepsilon)$. For $m$ sufficiently lar\-ge~$\phi_\bullet^m\psi_\bullet^n\big(\mathcal{U}(\alpha(\phi_\bullet),\varepsilon)\big)\subset\big(\mathcal{U}(\alpha(\phi_\bullet),\varepsilon/2)\big)\subsetneq\big(\mathcal{U}(\alpha(\phi_\bullet),\varepsilon)\big)$; hence $\phi_\bullet^m\psi_\bullet^n$ is hyperbolic. You can by this way build two hyperbolic maps whose sets of fixed points are distinct: this gives a contradiction with Lemma~\ref{lem:pas2hyp}. So for any $\varphi\in\mathrm{G}$ one has : $\alpha(\phi_\bullet)=\alpha(\varphi_\bullet)$; one can thus state the following result.

\begin{pro}
{\sl Let $\mathrm{G}$ be a solvable, non abelian subgroup of $\mathrm{Bir}(\mathbb{P}^2_\mathbb{C})$ that contains a twist $\phi$. Then 
\smallskip
\begin{itemize}
\item[$\bullet$] if $\phi$ is a Jonqui\`eres twist, then $\mathrm{G}$ preserves a rational fibration, that is up to birational conjugacy $\mathrm{G}$ is a subgroup of $\mathrm{PGL}(2,\mathbb{C}(y))\rtimes\mathrm{PGL}(2,\mathbb{C})$,
\smallskip
\item[$\bullet$] if $\phi$ is a Halphen twist, then $\mathrm{G}$ preserves an elliptic fibration.
\end{itemize} }
\end{pro}

\begin{rem}
Both cases are mutually exclusive.
\end{rem}

Note that if $\mathrm{G}$ is a subgroup of $\mathrm{Bir}(\mathbb{P}^2_\mathbb{C})$ that preserves an elliptic fibration, then $\mathrm{G}$ is countable~(\cite{Cantat:these}). Let us explain briefly why. A smooth rational projective surface $\mathcal{S}$ is a \textbf{\textit{Halphen surface}} if there exists an integer $m>0$ such that the linear system $\vert -mK_\mathcal{S}\vert$ is of dimension $1$, has no fixed component, and has no base point. The smallest positive integer for which $\mathcal{S}$ satisfies such a property is the \textbf{\textit{index}} of $\mathcal{S}$. If $\mathcal{S}$ is a Halphen surface of index $m$, then $K_\mathcal{S}^2=0$ and, by the genus formula, the linear system $\vert -mK_\mathcal{S}\vert$ defines a genus $1$ fibration $\mathcal{S}\to\mathbb{P}^1_\mathbb{C}$. This fibration is \textbf{\textit{relatively minimal}} in the sense that there is no $(-1)$-curve contained in a fiber. The following properties are equivalent:
\smallskip
\begin{itemize}
\item[$\bullet$] $\mathcal{S}$ is a Halphen surface of index $m$,
\smallskip
\item[$\bullet$] there exists an irreducible pencil of curves of degree $3m$ with $9$ base points of multiplicity $m$ in $\mathbb{P}^2_\mathbb{C}$ such that $\mathcal{S}$ is the blow-up of the $9$ base points and $\vert -mK_\mathcal{S}\vert$ is the proper transform of this pencil (the base points set may contain infinitely near points).
\end{itemize}
\smallskip
As a corollary of the classification of relatively minimal elliptic surfaces the relative minimal model of a rational elliptic surface is a Halphen surface of index $m$ (\cite[Chapter 2, \S 10]{IskovskikhShafarevich}). Up to conjugacy $\mathrm{G}$ is a subgroup of $\mathrm{Aut}(\mathcal{S})$ where $\mathcal{S}$ denotes a Halphen surface of index $m$. The action of $\mathrm{G}$ on $\mathrm{NS}(\mathcal{S})$ is almost faithful, and $\mathrm{G}$ is a discrete (it preserves the integral structure of $\mathrm{NS}(\mathcal{S})$) and virtually abelian (it preserves the intersection form and the class of the elliptic fibration) subgroup of $\mathrm{Aut}(\mathcal{S})$. So one has:

\begin{cor}
{\sl If $\mathrm{G}$ is an uncountable, solvable, non abelian subgroup of $\mathrm{Bir}(\mathbb{P}^2_\mathbb{C})$, then $\mathrm{G}$ doesn't contain a Halphen twist.}
\end{cor}

\begin{eg}
Let us come back to the example given in \S\ref{sec:intro}. If $\phi\in\mathrm{Bir}(\mathbb{P}^2_\mathbb{C})$ preserves a unique fibration that is rational then one can assume that up to birational conjugacy this fibration is given, in the affine chart $z=1$, by $y=$ cst. If $\phi$ preserves $y=$ cst fiberwise, then 
\smallskip
\begin{itemize}
\item[$\bullet$] $\phi$ is contained in a maximal abelian subgroup denoted $\mathrm{Ab}(\phi)$ that preserves $y=$ cst fiberwise (\cite{Deserti:autbir}),
\smallskip
\item[$\bullet$] the centralizer of $\phi$ is a finite extension of $\mathrm{Ab}(\phi)$ (\emph{see} \cite[Theorem B]{CerveauDeserti:centralisateurs}).
\end{itemize}
\smallskip
This allows us to establish that if $\phi$ preserves a fibration not fiberwise, then the centralizer of $\phi$ is virtually solvable. For instance if $\phi=\big(x+a(y),y+1\big)$ (resp. $\big(b(y)x,\beta y\big)$ or $\big(x+a(y),\beta y\big)$ with $\beta\in\mathbb{C}^*$ of infinite order) preserves a unique fibration, then the centralizer of $\phi$ is solvable and metabelian (\cite[Propositions 5.1 and 5.2]{CerveauDeserti:centralisateurs}).
\end{eg}

\subsection{Solvable groups with no hyperbolic map, and no twist}\label{subsec:elliptic}

Let $M$ be a smooth, irreducible, complex, projective variety of dimension $n$. Fix a K\"ahler form $\kappa$ on $M$. If $\ell$ is a positive integer, denote by $\pi_i\colon M^\ell\to M$ the projection onto the $i$-th factor. The manifold $M^\ell$ is then endowed with the K\"ahler form $\displaystyle\sum_{i=1}^\ell\pi_i^*\kappa$ which induces a K\"ahler metric. To any $\phi\in\mathrm{Bir}(M)$ one can associate its graph $\Gamma_\phi\subset M\times M$ defined as the Zariski closure of 
\[
\big\{(z,\phi(z))\in M\times M\,\vert\, z\in M\smallsetminus\mathrm{Ind}\,\phi\big\}. 
\]
By construction $\Gamma_\phi$ is an irreducible subvariety of $M\times M$ of dimension $n$. Both projections $\pi_1$, $\pi_2\colon M\times M\to~M$ restrict to a birational morphism $\pi_1$, $\pi_2\colon\Gamma_\phi\to M$.

\smallskip

The \textbf{\textit{total degree}} $\mathrm{tdeg}\,\phi$ of $\phi\in\mathrm{Bir}(M)$ is defined as the volume of $\Gamma_\phi$ with respect to the fixed metric on~$M\times M$:
\[
\mathrm{tdeg}\,\phi=\int_{\Gamma_\phi}\big(\pi_1^*\kappa+\pi_2^*\kappa\big)^n=\int_{M\smallsetminus\mathrm{Ind}\,\phi}\big(\kappa+\phi^*\kappa\big)^n.
\]

\smallskip

Let $d\geq 1$ be a natural integer, and set 
\[
\mathrm{Bir}_d(M)=\big\{\phi\in\mathrm{Bir}(M)\,\vert\,\mathrm{tdeg}\,\phi\leq d\big\}.
\]
A subgroup $\mathrm{G}$ of $\mathrm{Bir}(M)$ has \textbf{\textit{bounded degree}} if it is contained in $\mathrm{Bir}_d(M)$ for some $d\in\mathbb{N}^*$.

\smallskip

Any subgroup $\mathrm{G}$ of $\mathrm{Bir}(M)$ that has bounded degree can be regularized, that is up to birational conjugacy all indeterminacy points of all elements of $\mathrm{G}$ disappear simultaneously:

\begin{thm}[\cite{Weil}]
{\sl Let $M$ be a complex projective variety, and let $\mathrm{G}$ be a subgroup of $\mathrm{Bir}(M)$. If $\mathrm{G}$ has bounded degree, there exists a smooth, complex, projective variety $M'$, and a birational map $\psi\colon M'\dashrightarrow M$ such that $\psi^{-1}\mathrm{G}\psi$ is a subgroup of $\mathrm{Aut}(M')$.}
\end{thm}

The proof of this result can be found in \cite{HuckleberryZaitsev, Zaitsev}; an heuristic idea appears in \cite{Cantat:morphisms}.

\section{Applications}\label{sec:appl}

\subsection{Nilpotent subgroups of $\mathrm{Bir}(\mathbb{P}^2_\mathbb{C})$}

Let us recall that if $\mathrm{G}$ is a group, the \textbf{\textit{descending central series}} of $\mathrm{G}$ is defined by
\[
C^0\mathrm{G}=\mathrm{G}\qquad\qquad C^{n+1}\mathrm{G}=[\mathrm{G},C^n\mathrm{G}] \quad\forall\,n\geq 0.
\]
We say that $\mathrm{G}$ is \textbf{\textit{nilpotent}} if there exists $j\geq 0$ such that $C^j\mathrm{G}=\{\mathrm{id}\}$. If $j$ is the minimum non-negative number with such a property we say that $\mathrm{G}$ is of \textbf{\textit{nilpotent class}} $j$. Nilpotent subgroups of the Cremona group have been described:

\begin{thm}[\cite{Deserti:nilpotent}]\label{thm:nilp}
{\sl Let $\mathrm{G}$ be a nilpotent subgroup of $\mathrm{Bir}(\mathbb{P}^2_\mathbb{C})$. Then
\smallskip
\begin{itemize}
\item[$\bullet$] either\, $\mathrm{G}$ is up to finite index metabelian, 
\smallskip
\item[$\bullet$] or\, $\mathrm{G}$ is a torsion group.
\end{itemize}}
\end{thm}

We find an alternative proof of \cite[Lemma 4.2]{Deserti:nilpotent} for $\mathrm{G}$ infinite:

\begin{lem}
{\sl Let $\mathrm{G}$ be an infinite, nilpotent, non virtually abelian subgroup of $\mathrm{Bir}(\mathbb{P}^2_\mathbb{C})$. Then $\mathrm{G}$ does not contain a hyperbolic map.}
\end{lem}

\begin{proof}
The group $\mathrm{G}$ is also solvable. Assume by contradiction that $\mathrm{G}$ contains a hyperbolic map; then according to Theorem \ref{mainthm} up to birational conjugacy and finite index there exists $\Upsilon\subset\mathbb{C}^*\times\mathbb{C}^*$ infinite such that $\mathrm{G}$ is generated by $\phi=(x^py^q,x^ry^s)$ and
\[
\big\{(\alpha x,\beta y)\,\vert\,(\alpha,\,\beta)\in\Upsilon\big\}.
\]
The group $C^1\mathrm{G}$ contains 
\[
\big\{[\phi,(\alpha x,\beta y)]\,\vert\,(\alpha,\beta)\in\Upsilon\big\}=\big\{(\alpha^{p-1}\beta^qx,\alpha^r\beta^{s-1}y)\,\vert\,(\alpha,\beta)\in\Upsilon\big\}
\]
that is infinite since $\Upsilon$ is infinite. Suppose that $C^i\mathrm{G}$ contains the infinite set
\[
\big\{(\alpha^{\ell_i}\beta^{n_i}x,\alpha^{k_i}\beta^{m_i}y)\,\vert\,(\alpha,\beta)\in\Upsilon\big\}
\]
($\ell_i$, $n_i$, $k_i$ and $m_i$ are some functions in $p$, $q$, $r$ and $s$); then $C^{i+1}\mathrm{G}$ contains 
\[
\big\{(\alpha^{(p-1)\ell_i+qm_i}\beta^{(p-1)k_i+qn_i}x,\alpha^{r\ell_i+qm_i}\beta^{rk_i+(s-1)n_i}y)\,\vert\,(\alpha,\beta)\in\Upsilon\big\}
\]
that is still infinite.
\end{proof}

So any nilpotent and infinite subgroup of $\mathrm{Bir}(\mathbb{P}^2_\mathbb{C})$ falls in case 1., 2., 3. of Theorem \ref{mainthm}. If it falls in case 2. or 3. then $\mathrm{G}$ is virtually metabelian (\cite[Proof of Theorem 1.1.]{Deserti:nilpotent}). Finally if $\mathrm{G}$ falls in case 1., we can prove as in \cite{Deserti:nilpotent} that either  $\mathrm{G}$ is a torsion group, or $\mathrm{G}$ is virtually metabelian.

\subsection{Soluble length of a nilpotent subgroup of $\mathrm{Bir}(\mathbb{P}^2_\mathbb{C})$}

Let us recall the following statement due to Epstein and Thurston (\cite{EpsteinThurston}): let $M$ be a connected complex manifold. Let $\mathfrak{h}$ be a nilpotent Lie subalgebra of the complex vector space of rational vector fields on $M$. Then $\mathfrak{h}^{(n)}=\{0\}$ if $n\geq\dim M$; hence the solvable length of $\mathfrak{h}$ is bounded by the dimension of $M$. We have a similar statement in the context of birational maps; indeed a direct consequence of Theorem \ref{thm:nilp} is the following property: let $\mathrm{G}\subset\mathrm{Bir}(\mathbb{P}^2_\mathbb{C})$ be a nilpotent subgroup of $\mathrm{Bir}(\mathbb{P}^2_\mathbb{C})$ that is not a torsion group, then the soluble length of $\mathrm{G}$ is bounded by the dimension of $\mathbb{P}^2_\mathbb{C}$.

\subsection{Favre's question}\label{sec:favre}

In \cite{Favre} Favre asked few questions; among them there is the following: does any solvable, finitely generated subgroup $\mathrm{G}$ of $\mathrm{Bir}(\mathbb{P}^2_\mathbb{C})$ contain a subgroup $\mathrm{H}$ of finite index such that $[\mathrm{H},\mathrm{H}]$ is nilpotent ? We will prove that the answer is no if $[\mathrm{G},\mathrm{G}]$ is not a torsion group.

Take $\mathrm{G}$ a solvable and finitely generated subgroup of the Cremona group; besides suppose that $[\mathrm{G},\mathrm{G}]$ is not a torsion group. Assume that the answer of Favre's question is yes. Up to finite index one can assume that $[\mathrm{G},\mathrm{G}]$ is nilpotent. According to Theorem \ref{thm:nilp} the group $\mathrm{G}^{(1)}=[\mathrm{G},\mathrm{G}]$ is up to finite index metabelian; in other words up to finite index $\mathrm{G}^{(2)}=[\mathrm{G}^{(1)},\mathrm{G}^{(1)}]$ is abelian and so $\mathrm{G}^{(3)}=[\mathrm{G}^{(2)},\mathrm{G}^{(2)}]=\big\{\mathrm{id}\big\}$, {\it i.e.} the soluble length of $\mathrm{G}$ is bounded by $3$ up to finite index. Consider the subgroup
\[
\langle(x+y^2,y),\,(x(1+y),y),\,\left(x,\frac{y}{1+y}\right),\,(x,2y)\rangle
\]
of $\mathrm{Bir}(\mathbb{P}^2_\mathbb{C})$. It is solvable of length $4$ (\emph{see} \cite{MarteloRibon}): contradiction. 

\subsection{Baumslag-Solitar groups}

For any integers $m$, $n$ such that $mn\not= 0$, the Baumslag-Solitar group $\mathrm{BS}(m;n)$ is defined by the following presentation
\[
\mathrm{BS}(m;n)=\langle r,\,s\,\vert\,rs^mr^{-1}=s^n \rangle
\]

In \cite{BlancDeserti:degreegrowth} we prove that there is no embedding of $\mathrm{BS}(m;n)$ into $\mathrm{Bir}(\mathbb{P}^2_\mathbb{C})$ as soon as $\vert n\vert$, $\vert m\vert$, and $1$ are distinct; it corresponds exactly to the case $\mathrm{BS}(m;n)$ is not solvable. Indeed $\mathrm{BS}(m;n)$ is solvable if and only if $\vert m\vert = 1$ or $\vert n\vert = 1$ (\emph{see} \cite[Proposition A.6]{Souche}). 

\begin{pro}
{\sl Let $\rho$ be an embedding of $\mathrm{BS}(1;n)=\langle r,\,s\,\vert\,rsr^{-1}=s^n \rangle$, with $n\not=1$, into the Cremona group. Then
\smallskip
\begin{itemize}
\item[$\bullet$] the image of $\rho$ doesn't contain a hyperbolic map,
\smallskip
\item[$\bullet$] and
\[
\rho(s)=\big(x,y+1\big)\qquad \rho(r)=\big(\nu(x),n(y+a(x))\big)
\]
with $\nu\in\mathrm{PGL}(2,\mathbb{C})$ and $a\in\mathbb{C}(x)$. 
\end{itemize}}
\end{pro}

\begin{proof}
According to \cite[Proposition 6.2, Lemma 6.3]{BlancDeserti:degreegrowth} one gets that $\rho(s)=(x,y+1)$ and $\rho(r)=\big(\nu(x),n(y+a(x))\big)$ for some $\nu\in\mathrm{PGL}(2,\mathbb{C})$ and $a\in\mathbb{C}(x)$. 

Furthermore $\rho(s)$ can neither be conjugate to an automorphism of the form $(\alpha x,\beta y)$ (\emph{see} \cite{Blanc:conjugacyclasses}), nor to a hyperbolic birational map of the form $(\gamma x^py^q,\delta x^ry^s)$ with $\left[\begin{array}{cc} p & q \\ r & s\end{array}\right]\in\mathrm{GL}(2,\mathbb{Z})$ of spectral radius $>1$. As a consequence Proposition \ref{pro:1hyp} implies that $\rho(\mathrm{BS}(1;n))$ does not contain a hyperbolic birational map.
\end{proof}

\vspace{8mm}

\bibliographystyle{plain}
\bibliography{biblio}

\begin{thebibliography}{10}

\bibitem{Blanc:conjugacyclasses}
J.~Blanc.
\newblock Conjugacy classes of affine automorphisms of {$\Bbb K^n$} and linear
  automorphisms of {$\Bbb P^n$} in the {C}remona groups.
\newblock {\em Manuscripta Math.}, 119(2):225--241, 2006.

\bibitem{BlancCantat}
J.~Blanc and S.~Cantat.
\newblock Dynamical degrees of birational transformations of projective
  surfaces, \texttt{arXiv:1307.0361}.
\newblock {\em J. Amer. Math. Soc.}, to appear.

\bibitem{BlancDeserti:degreegrowth}
J.~Blanc and J.~D{\'e}serti.
\newblock Degree growth of birational maps of the plane.
\newblock {\em Ann. Sc. Norm. Super. Pisa Cl. Sci. (5)}, To appear.

\bibitem{BridsonHaefliger}
M.~R. Bridson and A.~Haefliger.
\newblock {\em Metric spaces of non-positive curvature}, volume 319 of {\em
  Grundlehren der Mathematischen Wissenschaften}.
\newblock Springer-Verlag, Berlin, 1999.

\bibitem{Cantat:these}
S.~Cantat.
\newblock {\em Dynamique des automorphismes des surfaces complexes compactes}.
\newblock PhD thesis, \'Ecole Normale Sup\'erieure de Lyon, 1999.

\bibitem{Cantat:tits}
S.~Cantat.
\newblock Sur les groupes de transformations birationnelles des surfaces.
\newblock {\em Ann. of Math. (2)}, 174(1):299--340, 2011.

\bibitem{Cantat:morphisms}
S.~Cantat.
\newblock Morphisms between {C}remona groups, and characterization of rational
  varieties.
\newblock {\em Compos. Math.}, 150(7):1107--1124, 2014.

\bibitem{CantatLamy}
S.~Cantat and S.~Lamy.
\newblock Normal subgroups in the {C}remona group.
\newblock {\em Acta Math.}, 210(1):31--94, 2013.
\newblock With an appendix by Yves de Cornulier.

\bibitem{CerveauDeserti:centralisateurs}
D.~Cerveau and J.~D{\'e}serti.
\newblock Centralisateurs dans le groupe de {J}onqui\`eres.
\newblock {\em Michigan Math. J.}, 61(4):763--783, 2012.

\bibitem{Cornulier}
Y.~Cornulier.
\newblock Nonlinearity of some subgroups of the planar {C}remona group.
\newblock {\em Unpublished manuscript, 2013,
  \url{http://www.normalesup.org/~cornulier/crelin.pdf}}.

\bibitem{delaHarpe}
P.~de~la Harpe.
\newblock {\em Topics in geometric group theory}.
\newblock Chicago Lectures in Mathematics. University of Chicago Press,
  Chicago, IL, 2000.

\bibitem{DelzantPy}
T.~Delzant and P.~Py.
\newblock K\"ahler groups, real hyperbolic spaces and the {C}remona group.
\newblock {\em Compos. Math.}, 148(1):153--184, 2012.
\newblock With an appendix by S. Cantat.

\bibitem{Deserti:autbir}
J.~D{\'e}serti.
\newblock Sur les automorphismes du groupe de {C}remona.
\newblock {\em Compos. Math.}, 142(6):1459--1478, 2006.

\bibitem{Deserti:nilpotent}
J.~D{\'e}serti.
\newblock Sur les sous-groupes nilpotents du groupe de {C}remona.
\newblock {\em Bull. Braz. Math. Soc. (N.S.)}, 38(3):377--388, 2007.

\bibitem{DillerFavre}
J.~Diller and C.~Favre.
\newblock Dynamics of bimeromorphic maps of surfaces.
\newblock {\em Amer. J. Math.}, 123(6):1135--1169, 2001.

\bibitem{DolgachevIskovskikh}
I.~V. Dolgachev and V.~A. Iskovskikh.
\newblock Finite subgroups of the plane {C}remona group.
\newblock In {\em Algebra, arithmetic, and geometry: in honor of {Y}u. {I}.
  {M}anin. {V}ol. {I}}, volume 269 of {\em Progr. Math.}, pages 443--548.
  Birkh\"auser Boston, Inc., Boston, MA, 2009.

\bibitem{EpsteinThurston}
D.~B.~A. Epstein and W.~P. Thurston.
\newblock Transformation groups and natural bundles.
\newblock {\em Proc. London Math. Soc. (3)}, 38(2):219--236, 1979.

\bibitem{Favre}
C.~Favre.
\newblock Le groupe de {C}remona et ses sous-groupes de type fini.
\newblock {\em Ast\'erisque}, (332):Exp. No. 998, vii, 11--43, 2010.
\newblock S{\'e}minaire Bourbaki. Volume 2008/2009. Expos{\'e}s 997--1011.

\bibitem{GhysdelaHarpe}
{\'E}.~Ghys and P.~de~la Harpe, editors.
\newblock {\em Sur les groupes hyperboliques d'apr\`es {M}ikhael {G}romov},
  volume~83 of {\em Progress in Mathematics}.
\newblock Birkh\"auser Boston, Inc., Boston, MA, 1990.
\newblock Papers from the Swiss Seminar on Hyperbolic Groups held in Bern,
  1988.

\bibitem{Gizatullin}
M.~H. Gizatullin.
\newblock Rational {$G$}-surfaces.
\newblock {\em Izv. Akad. Nauk SSSR Ser. Mat.}, 44(1):110--144, 239, 1980.

\bibitem{HuckleberryZaitsev}
A.~Huckleberry and D.~Zaitsev.
\newblock Actions of groups of birationally extendible automorphisms.
\newblock In {\em Geometric complex analysis ({H}ayama, 1995)}, pages 261--285.
  World Sci. Publ., River Edge, NJ, 1996.

\bibitem{IskovskikhShafarevich}
V.~A. Iskovskikh and I.~R. Shafarevich.
\newblock Algebraic surfaces.
\newblock In {\em Algebraic geometry, {II}}, volume~35 of {\em Encyclopaedia
  Math. Sci.}, pages 127--262. Springer, Berlin, 1996.

\bibitem{KargapolovMerzljakov}
M.~I. Kargapolov and Ju.~I. Merzljakov.
\newblock {\em Fundamentals of the theory of groups}, volume~62 of {\em
  Graduate Texts in Mathematics}.
\newblock Springer-Verlag, New York-Berlin, 1979.
\newblock Translated from the second Russian edition by Robert G. Burns.

\bibitem{MarteloRibon}
M.~Martelo and J.~Rib{\'o}n.
\newblock Derived length of solvable groups of local diffeomorphisms.
\newblock {\em Math. Ann.}, 358(3-4):701--728, 2014.

\bibitem{Souche}
E.~Souche.
\newblock {\em Quasi-isom\'etrie et quasi-plans dans l'\'etude des groupes
  discrets}.
\newblock PhD thesis, Universit\'e de Provence, 2001.

\bibitem{Tits}
J.~Tits.
\newblock Free subgroups in linear groups.
\newblock {\em J. Algebra}, 20:250--270, 1972.

\bibitem{Weil}
A.~Weil.
\newblock On algebraic groups of transformations.
\newblock {\em Amer. J. Math.}, 77:355--391, 1955.

\bibitem{Wright}
D.~Wright.
\newblock Abelian subgroups of {${\rm Aut}_{k}(k[X,\,Y])$} and applications to
  actions on the affine plane.
\newblock {\em Illinois J. Math.}, 23(4):579--634, 1979.

\bibitem{Zaitsev}
D.~Zaitsev.
\newblock Regularization of birational group operations in the sense of {W}eil.
\newblock {\em J. Lie Theory}, 5(2):207--224, 1995.

\end{thebibliography}
\nocite{}

\end{document}